\documentclass[a4paper]{amsart}
  \usepackage{amsmath, amsfonts, amssymb, amsthm, bbm}
  \usepackage{enumerate}
  \usepackage{ifthen}

\theoremstyle{definition}
\newtheorem{cor}{Corollary}[section]
\newtheorem{dfn}[cor]{Definition} 

\newtheorem{lem}[cor]{Lemma}
\newtheorem{prop}[cor]{Proposition}

\newtheorem{thm}[cor]{Theorem}

\newcommand{\N}{\mathbb{N}}

\newcommand{\R}{\mathbb{R}}

\newcommand{\Alg}{\mathrm{Alg}}
\newcommand{\alg}{\mathrm{alg}}
\newcommand{\Eig}{\mathrm{Eig}}

\newcommand{\GL}{\mathrm{GL}}

\newcommand{\Mat}{\mathrm{Mat}}
\newcommand{\ve}{\mathrm{vec}}
\newcommand{\Ve}{\mathrm{Vec}}

\begin{document}

\title{Counting finite-dimensional algebras over finite fields}
\author{Nikolaas D. Verhulst}
\address{TU Dresden\\
	Fachrichtung Mathematik\\
	Insitut f\"ur algebra\\
	01062 Dresden, Germany}
	\email{nikolaas\_damiaan.verhulst@tu-dresden.de}

\begin{abstract}
In this paper, we describe an elementary method for counting the number of non-isomorphic algebras of a fixed dimension over a given finite field. We show how this method works for the explicit example of $2$-dimensional algebras over the field $\mathbb{F}_{2}$.
\end{abstract}

\maketitle

\section*{Introduction}

Classifying finite dimensional algebras over a given field is usually a very hard problem. The first general result was a classification of $2$-dimensional algebras over the base field $\R$, which appeared in 1992 (\cite{AlthoenHansen}). This was generalised by Petersson (\cite{Petersson}), who managed to give a full classification of $2$-dimensional algebras over an arbitrary base field. The methods employed in these papers are quite involved and rely on a large amount of previous work by many illustrious authors. 

Our aim in this paper is to give perhaps not a classification but at least a way to compute the exact number of non-isomorphic $n$-dimensional algebras over a fixed finite field by elementary means. Indeed, nothing more complicated than linear algebra and some very basic results about group actions will be needed: we describe isomorphism classes of $n$-dimensional $K$-algebras as orbits of a certain $\GL_{n}(K)$-action on $\Mat_{n}(K)^{n}$ and use a basic result about group actions to count these orbits. In the last section, we work out a concrete example.

\section{Notation and basics}

Fix a field $K$. In this article, an \emph{algebra} is understood to be a $K$-vector space $A$ equipped with a multiplication, i.e. a bilinear map $A\times A\to A$. If $a,b$ are in $K$, we will write $ab$ for the image of $(a,b)$ under this map. We do not assume algebras to have a unit or to be associative. By the \emph{dimension} of an algebra we mean its dimension as a $K$-vector space. Two algebras $A$ and $A'$ will be called \emph{isomorphic} if there exists a $K$-linear bijection $f:A\to A'$ with $f(ab)=f(a)f(b)$ for all $a,b$ in $A$. The isomorphism class of an algebra $A$ will be denoted by $[A]$. For $n\in\N$, we define $\Alg_{n}(K)$ to be the set of isomorphism classes of $n$-dimensional algebras. 

Given a vector $\mathcal{M}=(M_{i})_{i=1,...,n}$ of $n$ $n\times n$-matrices over $K$, we can define an algebra $\alg(\mathcal{M})$ which is $K^{n}$ as a $K$-vector space and for which multiplication is defined to be the unique bilinear map $K^{n}\times K^{n}\to K^{n}$ with\[e_{i}e_{j}=\sum_{k}(M_{i})_{jk}e_{k}\]
where the $e_{i}$ are the canonical basis vectors of $K^{n}$. Intuitively, this means that multiplying an element $a\in A$ with $e_{i}$ is multiplying the coordinate vector of $a$ (with respect to the canonical basis) with $M_{i}$ and interpreting the result again as a coordinate vector (with respect to the canonical basis). This allows us to define the map \[\alg:\Mat_{n}(K)^{n}\to\Alg_{n}(K),\mathcal{M}\mapsto [\alg(\mathcal{M})]\]
which will play an important role in this paper.

\begin{prop}\label{surjectivity}
The map $\alg$ defined above is surjective.
\end{prop}
\begin{proof}
Let $A$ be an $n$-dimensional algebra with basis $a_{1},...,a_{n}$. There are $\alpha_{ij,k}$ in $K$ such that $a_{i}a_{j}=\sum_{k}\alpha_{ij,k}a_{k}$ for all $1\leq i,j\leq n$. Define the matrix $M_{i}$ by putting $(M_{i})_{jk}=\alpha_{ij,k}$ and set $\mathcal{M}=(M_{i})_{i=1,...,n}\in\Mat_{n}(K)^{n}$. There is a unique linear map $\alg(\mathcal{M})\to A,e_{i}\mapsto a_{i}$ which is clearly bijective and which, by construction, preserves multiplication. Hence $[A]=[\alg(\mathcal{M})]$.
\end{proof}

\noindent On the other hand, $\alg$ is clearly not injective, since for any $\mathcal{M}=(M_{i})_{i=1,...,n}$ and $\alpha\in K^{*}$, for example, we have $[\alg(\mathcal{M})]=[\alg(\alpha\mathcal{M})]$. 

\section{A group action on $\Mat_{n}(K)^{n}$}

Recall that for a given set $X$ and a group $G$ with neutral element $e$, a (right) \emph{$G$-action} on $X$ is a map $\phi:X\times G\to X$ such that 
\begin{enumerate}[(1)]
	\item $\phi(x,e)=x$ for any $x\in X$,
	\item $\phi(x,gg')=\phi(\phi(x,g),g')$ for all $g,g'\in G$, $x\in X$.
\end{enumerate}
If $\phi$ is a $G$-action on $X$, then the \emph{$\phi$-orbit} of an element $x\in X$ is the set $G(x)=\left\{\phi(x,g)\mid g\in G\right\}$. The set of orbits is denoted by $X/G$. The \emph{fixpoints} of a $g\in G$ are the elements of $X^{g}=\left\{x\in X\mid \phi(x,g)=x\right\}$.

\begin{lem}\label{groupaction}
The map 
\[\phi:\Mat_{n}(K)^{n}\times \GL_{n}(K)\to \Mat_{n}(K)^{n},\left(\begin{bmatrix}M_{1}\\\vdots\\M_{n}\end{bmatrix},G\right)\mapsto \begin{bmatrix}G^{-1}\sum_{i}G_{i1}M_{i}G\\\vdots\\G^{-1}\sum_{i}G_{in}M_{i}G\end{bmatrix}\]
is a $\GL_{n}(K)$-action on $\Mat_{n}(K)^{n}$.
\end{lem}
\begin{proof}
It is clear that $\phi(\mathcal{M},\mathbbm{1}_{n})=\mathcal{M}$ for all $\mathcal{M}$ in $\Mat_{n}(K)^{n}$. Take $\mathcal{M}=(M_{i})_{i=1,...,n}$ in $\Mat_{n}(K)^{n}$ and $G,G'$ in $\GL_{n}(K)$. We have to show $\phi(\mathcal{M},GG')=\phi(\phi(\mathcal{M},G),G')$. The term on the right is 
\begin{align*}
\phi\left(\begin{bmatrix}
G^{-1}\sum_{i}G_{i1}M_{i}G\\\vdots\\G^{-1}\sum_{i}G_{in}M_{i}G
\end{bmatrix},G'\right)&=\begin{bmatrix}
G'^{-1}\sum_{j}G'_{j1}\left(G^{-1}\sum_{i}G_{ij}M_{i}G\right)G'\\\vdots\\
G'^{-1}\sum_{j}G'_{jn}\left(G^{-1}\sum_{i}G_{ij}M_{i}G\right)G'
\end{bmatrix}\\&=
\begin{bmatrix}
(GG')^{-1}\sum_{i}(GG')_{i1}M_{i}GG'\\
\vdots\\
(GG')^{-1}\sum_{i}(GG')_{in}M_{i}GG'
\end{bmatrix}
\end{align*}
which is the term on the left.
\end{proof}

\begin{lem}\label{orbitisisoclass}
Two elements $\mathcal{M},\mathcal{M}'$ of $\Mat_{n}(K)^{n}$ are in the same $\phi$-orbit if and only if $\alg(\mathcal{M})$ and $\alg(\mathcal{M}')$ are isomorphic. 
\end{lem}
\begin{proof}
Assume $\alg(\mathcal{M})$ and $\alg(\mathcal{M}')$ to be isomorphic for some $\mathcal{M}=(M_{i})_{i=1,...,n}$ and $\mathcal{M}'=(M'_{i})_{i=1,...,n}$ in $\Mat_{n}(K)^{n}$. Take an isomorphism $f:\alg(\mathcal{M})\to \alg(\mathcal{M}')$. Since $\alg(\mathcal{M})$ and $\alg(\mathcal{M}')$, considered as $K$-vector spaces, are just $K^{n}$, there must be a $G\in \GL_{n}(K)$ such that $f(x)$ is just $Gx$ for all $x\in \alg(\mathcal{M}')$. As $f$ is an isomorphism, we find 
\begin{align*}
G\sum_{i}x_{i}M_{i}y&=f(xy)=f(x)f(y)=Gx\cdot Gy=\sum_{i}\left(\sum_{j}G_{ij}x_{j}M_{i}'\right)Gy
\end{align*}
for arbitrary $x,y\in\alg(\mathcal{M}')$. In particular, if $x=e_{l}$, we find $GM_{l}y=\sum_{i}G_{il}M_{i}'Gy$ for all $y$, so $M_{l}=G^{-1}\sum_{i}G_{il}M_{i}'G$ showing $\phi(\mathcal{M}',G)=\mathcal{M}$.

Suppose now that, for given $\mathcal{M}$ and $\mathcal{M}'$ in $\Mat_{n}(K)^{n}$, there is some $G\in\GL_{n}(K)$ with $\phi(\mathcal{M'},G)=\mathcal{M}$. $G$ induces a function $f:K^{n}\to K^{n},x\mapsto Gx$ which is bijective as $G$ is invertible. To prove that $f$ is an isomorphism between $\alg(\mathcal{M})$ and $\alg(\mathcal{M}')$, it suffices to show $f(e_{i}e_{j})=f(e_{i})f(e_{j})$ for all $i$, $j$ since $f$ is linear. We find 
\begin{align*}
f(e_{i})f(e_{j})=(Ge_{i})(Ge_{j})&=\left(\sum_{k}G_{ki}M'_{k}\right)(Ge_{j})=\left(\sum_{k}G_{ki}M'_{k}G\right)e_{j}\\
&=GG^{-1}\left(\sum_{k}G_{ki}M'_{k}G\right)e_{j}=GM_{i}e_{j}= f(e_{i}e_{j}),
\end{align*} 
the penultimate equality following from $\phi(\mathcal{M}',G)=\mathcal{M}$.
\end{proof}

\section{Counting orbits}
\emph{From now on, we assume $K$ to be a finite field with $q$ elements.} As a consequence of proposition \ref{orbitisisoclass}, we find that $\alg$ induces a well-defined, injective map \[\overline{\alg}:\Mat_{n}(K)^{n}/\GL_{n}(K)\to\Alg_{n}(K),G(\mathcal{M})\mapsto[\alg(\mathcal{M})]\]
which is also surjective by \ref{surjectivity}. The number of isomorphism classes of $n$-dimensional $K$-algebras therefore equals the number of $\phi$-orbits of $\Mat_{n}(K)^{n}$. The following well-known result result from the theory of group actions will help us count the latter:

\begin{prop}[Burnside's lemma]\label{Burnside}
Suppose $\phi$ is an action of a finite group $G$ on a finite set $X$.  Then
\[|X/G|=\frac{1}{|G|}\sum_{g\in G}|X^{g}|.\]
\end{prop}
\begin{proof}
Cfr. e.g. \cite{Rotman}, p.58.
\end{proof}

To use this lemma, we need to know the number of fixpoints of a given invertible matrix $M$. For that, we need the following definition:

\begin{dfn}
For a matrix $M\in\Mat_{k\times l}(K)$, the \emph{vectorisation} of $M$ is the vector $\ve(M)\in K^{kl}$ obtained by stacking the columns of $M$, the first column being on top. For an element $\mathcal{M}=(M_{i})_{i=1,...,n}\in\Mat_{n}(K)^{n}$, we write $\Ve(\mathcal{M})$ for the single vector consisting of the vectorisations of all the $M_{i}$. For more on the vectorisation operation, we refer to \cite{VecTrick}.
\end{dfn}

\begin{lem}\label{fixpoints}
For an invertible matrix $M$, we have
\[|X^{M}|=q^{\dim\Eig_{1}(M^{T}\otimes M^{T}\otimes M^{-1})}\]
where $\Eig_{1}(A)$ denotes the eigenspace of the matrix $A$ with eigenvalue $1$.
\end{lem}
\begin{proof}
Suppose $\mathcal{N}=(N_{i})_{i=1,...,n}$ is a fixpoint of $M$, i.e. 
\[
N_{l}=M^{-1}\sum_{i}M_{il}N_{i}M\quad \text{for all $l$}.\tag{$\dagger$}\label{fixpointeq}
\] It is known (see e.g. \cite{VecTrick}) that, for arbitrary $A,B$ in $\Mat_{n}(K)$, $M^{-1}AM=B$ is equivalent to $\ve(B)=(M^{T}\otimes M^{-1})\ve(A)$. From this, we conclude: \[\Ve(M^{-1}N_{i}M)_{i=1,...,n}=( \mathbbm{1}_{n}\otimes M^{T}\otimes M^{-1})\Ve(\mathcal{N}).\] We have furthermore that $\Ve((\sum_{i} M_{il}O_{i})_{l=1,...,n})=(M^{T}\otimes \mathbbm{1}_{n}\otimes\mathbbm{1}_{n})\Ve(\mathcal{O})$ for any $\mathcal{O}=(O_{i})_{i=1,...,n}\in\Mat_{n}(K)^{n}$. Consequently, \eqref{fixpointeq} is equivalent to 
\begin{align*}
\Ve(\mathcal{N})&=(M^{T}\otimes\mathbbm{1}_{n}\otimes\mathbbm{1}_{n})(\mathbbm{1}_{n}\otimes M^{T}\otimes M^{-1})\Ve(\mathcal{N})\\
&= (M^{T}\otimes M^{T}\otimes M^{-1})\Ve(\mathcal{N}),
\end{align*} 
so $\mathcal{N}$ is a fixpoint of $M$ if and only if $\Ve(\mathcal{N})$ is an eigenvector of $M^{T}\otimes M^{T}\otimes M^{-1}$ with eigenvalue $1$.

 %In general, the eigenvectors of $A\otimes B$ for any matrices $A$ and $B$ are of the form $v\otimes w$ where $v$ is an eigenvector of $A$ and $w$ is an eigenvector of $B$. Moreover, if the eigenvalues associated to $v$ and $w$ are $\lambda $ and $\mu$ respectively, then $v\otimes w$ has eigenvalue $\lambda\mu$. Of course, if $\lambda_{1},...,\lambda_{k}$ are the eigenvalues of a matrix $A$, then $\lambda_{1},...,\lambda_{k}$ are also the eigenvalues of $A^{T}$ while $\lambda_{1}^{-1},...,\lambda_{k}^{-1}$ are the eigenvalues of $A^{-1}$.
 %
 %$M^{T}\otimes M^{T}\otimes M^{-1}$ therefore has an eigenvector with eigenvalue $1$ if and only if there are eigenvalues $\lambda, \mu$ of $M$ such that $\lambda\mu$ is an eigenvalue of $M$ as well. In fact, we have 
%\begin{align*}
 %\dim\Eig_{1}(M^{T}\otimes M^{T}\otimes M^{-1})=\sum_{(\alpha,\beta)\in \overline{K}^{2}}\dim\Eig_{\alpha}(M)\dim\Eig_{\beta}(M)\dim\Eig_{\alpha\beta}(M)
 %\end{align*}
%which concludes the proof. 
\end{proof}

\begin{thm}\label{numberofisoclasses}
The number of non-isomorphic $n$-dimensional $K$-algebras is \[|\Alg_{n}(K)|=\frac{1}{|\GL_{n}(K)|}\sum_{M\in\GL_{n}(K)}q^{\dim\Eig_{1}(M^{T}\otimes M^{T}\otimes M^{-1})}.\tag{$\ddagger$}\label{isoclasseq}\]
\end{thm}
\begin{proof}
By lemma \ref{orbitisisoclass}, the number of non-isomorphic $n$-dimensional $k$-algebras is the number of $\phi$-orbits. By \ref{Burnside} and \ref{fixpoints}, this is equal to the given formula.
\end{proof}

\noindent Note that, for any invertible $M$, we have \[\dim\Eig_{1}(M^{T}\otimes M^{T}\otimes M^{-1})\leq n^{3}\] and that equality holds precisely when $M=\mathbbm{1}_{n}$. This gives an obvious lower bound on the number of non-isomorphic $n$-dimensional algebras over the field with $q$ elements:\[\frac{q^{n^{3}}}{(q^{n}-1)\cdots(q^{n}-q^{n-1})}.\] For a fixed $n$, this bound will become sharp as $q$ tends to infinity.

\section{Example: the case $n=2$, $q=2$}

Any element of $\GL_{2}(K)$ has, counting multiplicities, two eigenvalues in the algebraic closure of $K$. Clearly, either both or none are elements of $K$ which makes counting invertible matrices with eigenvalues in $K$ considerably easier. Indeed, the only possible Jordan normal forms for a $2\times 2$ matrix are
\[J_{1}=\begin{bmatrix}\alpha&0\\0&\beta\end{bmatrix}\quad\text{and}\quad J_{2}=\begin{bmatrix}\alpha&1\\0&\alpha\end{bmatrix}\]for some $\alpha,\beta$ in the algebraic closure of $K$. Since $M=SJS^{-1}$, $M'=S'J'S'^{-1}$ implies $M\otimes M'=(S\otimes S')(J\otimes J')(S^{-1}\otimes S'^{-1})$ and since the Jordan normal form of $M$ is also the Jordan normal form of $M^{T}$, it follows that every $M^{T}\otimes M^{T}\otimes M^{-1}$ must be conjugate either to  
\[M_{\alpha,\beta}=\begin{bmatrix}\alpha&&&&&&&\\&\alpha&&&&&&\\&&\alpha&&&&&\\&&&\alpha^{2}\beta^{-1}&&&&\\&&&&\alpha^{-1}\beta ^{2}&&&\\&&&&&\beta&&\\&&&&&&\beta&\\&&&&&&&\beta\end{bmatrix}\]
or to
\[N_{\alpha}=\begin{bmatrix}\alpha&-1&1&-\alpha^{-1}&1&-\alpha^{-1}&\alpha^{-1}&-\alpha^{-2}\\&\alpha&&1&&1&&-\alpha^{-1}\\&&\alpha&-1&&&1&-\alpha^{-1}\\&&&\alpha&&&&1\\&&&&\alpha&-1&1&-\alpha^{-1}\\&&&&&\alpha&&1\\&&&&&&\alpha&-1\\&&&&&&&\alpha\end{bmatrix}.\]
for some $\alpha,\beta\in \overline{K}$. Note that $\dim\Eig_{1}(N_{1})=3$ unless the characteristic of $K$ is $2$, in which case $\dim\Eig_{1}(N_{1})=4$. If $\alpha\neq1$, we obviously have $\dim\Eig_{1}(N_{\alpha})=0$. For $M_{\alpha,\beta}$, the dimension of the eigenspace associated to $1$ depends heavily on $\alpha$ and $\beta$, ranging from $8$ if $\alpha=\beta=1$ to $0$ if $1\notin\left\{\alpha,\beta,\alpha^{2}\beta^{-1},\alpha^{-1}\beta^{2}\right\}$. 

We will do the computations explicitly for the concrete example of $K=\mathbb{F}_{2}$. There are $6$ invertible matrices, namely
\[\mathbbm{1}_{2},\quad \begin{bmatrix}0&1\\1&0\end{bmatrix},\quad\begin{bmatrix}1&0\\1&1\end{bmatrix},\quad \begin{bmatrix}1&1\\0&1\end{bmatrix},\quad \begin{bmatrix}0&1\\1&1\end{bmatrix},\quad \begin{bmatrix}1&1\\1&0\end{bmatrix}.\]
The identity obviously yields a contribution of $2^{8}$. The next three are conjugate to $J_{2}$ with $\alpha=1$, therefore yielding a contribution of $2^{4}$ each. The last two have no eigenvalues over $K$. Their eigenvalues are the roots $t_{1},t_{2}$ of the polynomial $x^{2}+x+1$. As these roots satisfy $t_{1}^{2}=t_{2},t_{2}^{2}=t_{1}$, both matrices give a contribution of $2^2$. This gives a total of $2^{8}+3\cdot 2^{4}+2\cdot 2^{2}=312$ which divided by the total number of invertible matrices gives $312/6=52$. This number fits the formulae which were obtained, using completely different methods, by Petersson and Scherer in \cite{PeterssonScherer}.

\section{Outlook}

Theorem \ref{numberofisoclasses} suggest the following question: how many invertible $n\times n$-matrices $M$ have \[\dim\Eig(M^{T}\otimes M^{T}\otimes M^{-1},1)=k\] for a given $k\in\N$? If $q$ and $n$ are fixed, this is a finite problem and can therefore be calculated, but this is rather tedious and time-consuming. Having a closed formula in $q$ and $n$ would be nice.

On the algebraic side, it would be interesting to see whether the method described in this paper can also be used to count certain subclasses of algebras, like alternating algebras, associative algebras, or division algebras.

\end{document}